\documentclass[11pt]{amsart}

\usepackage{amssymb,amsfonts}
\usepackage{hyperref}
\usepackage{mdwlist}
\usepackage{amssymb,textcomp}
\usepackage[dvipsnames]{xcolor}
\usepackage{mathtools}
\usepackage{tikz}
\usetikzlibrary{positioning}
\usepackage{enumerate}
\usepackage{enumitem}

\usepackage[utf8]{inputenc}
\usepackage[T1]{fontenc}

\usepackage[mathscr]{eucal}

\usepackage{amssymb,amsfonts}
\usepackage{hyperref}
\usepackage{mdwlist}
\usepackage{amssymb,textcomp}
\usepackage[dvipsnames]{xcolor}
\usepackage{mathtools}
\usepackage{tikz}
\usetikzlibrary{positioning}
\usepackage{enumerate}
\usepackage{enumitem}

\usepackage[utf8]{inputenc}
\usepackage[T1]{fontenc}

\usepackage[mathscr]{eucal}

\usepackage{hyperref}
 \hypersetup{
     colorlinks=true,
     linktocpage=true,
     linkcolor=red,
     filecolor=blue,
     citecolor = blue,
     urlcolor=cyan,
     }

\usepackage[a4paper, twoside=false, vmargin={2cm,3cm}, includehead]{geometry}

\usepackage{comment}
\newtheorem{lemma}{Lemma}[section]
\newtheorem{theorem}[lemma]{Theorem}
\newtheorem*{theorem*}{Theorem}

\newtheorem{corollary}[lemma]{Corollary}

\newtheorem{proposition}[lemma]{Proposition}
\newtheorem*{proposition*}{Proposition}

\newtheorem{conjecture}[lemma]{Conjecture}

\newtheorem*{problem*}{Problem}

\theoremstyle{definition}

\newtheorem*{claim*}{Claim}

\newtheorem{definition}[lemma]{Definition}

\newtheorem{example}[lemma]{Example}

\newtheorem*{remark}{Remark}
\newtheorem*{remarks}{Remarks}

\DeclareMathOperator*{\E}{\mathbb{E}}

\newcommand{\C}{{\mathbb C}}

\newcommand{\N}{{\mathbb N}}
\renewcommand{\P}{{\mathbb P}}
\newcommand{\Q}{{\mathbb Q}}
\newcommand{\R}{{\mathbb R}}
\renewcommand{\S}{\mathbb{S}}
\newcommand{\T}{{\mathbb T}}
\newcommand{\Z}{{\mathbb Z}}

\newcommand{\e}{\mathrm{e}}

\newcommand{\norm}[1]{\left\Vert #1\right\Vert}

\newcommand{\floor}[1]{{\left \lfloor #1 \right \rfloor}}

\usepackage[dvipsnames]{xcolor}
\newcommand{\FM}[1]{{\color{violet}{{#1}}}}
\newcommand{\AK}[1]{{\color{blue}{{#1}}}}

\title[Furstenberg systems of certain sequences of superpolynomial growth]{Furstenberg systems of certain sequences of superpolynomial growth}
\author{Andreu Ferr\'e Moragues and Andreas Koutsogiannis}

\address[Andreu Ferr\'e Moragues]{Department of Mathematical Analysis and Applied Mathematics, Faculty of Mathematics, Complutense University of Madrid, 28040 Madrid, Spain
}
\email{anferre@ucm.es}

\address[Andreas Koutsogiannis]{
Department of Mathematics, Aristotle University of Thessaloniki, Thessaloniki 54124, Greece}
\email{akoutsogiannis@math.auth.gr}

\thanks{The research project is implemented in the framework of H.F.R.I call ``3rd Call for H.F.R.I.’s
Research Projects to Support Faculty Members \& Researchers'' (H.F.R.I. Project Number: 24979).}



\begin{document}

\maketitle
\begin{abstract}
We give examples of sequences defined by smooth functions of intermediate growth, and we study the Furstenberg systems that model their statistical behavior. In particular, we show that the systems are Bernoulli. We do so by studying exponential sums that reflect the strong equidistribution properties of said sequences. As a by-product of our approach, we also get some convergence results.
\end{abstract}

\tableofcontents

\section{Introduction}
The seminal proof of Szemer\'edi's theorem using ergodic theory due to Furstenberg \cite{Fu77} initiated a new method to study problems in combinatorics: through the use of his, now classical, correspondence principle, one can translate the problem of finding patterns in large subsets $E \subseteq \Z$ into recurrence properties on a measure preserving system $(X,\mathcal{B},\mu,T)$ for a given set $A \in \mathcal{B}$ with $\mu(A)>0$.

It was soon realized that one can actually use a slight variant of the aforementioned principle to study statistical properties of bounded sequences $a: \N \to \C$ (and more generally, properties of finite collections of bounded sequences $a_1,\dots,a_r: \N \to \C$) via convenient dynamical models.

For example (cf. \cite[Theorem~3.13]{Fu81}), the statistical behavior of the sequence $(n^2\alpha)_n$ on the torus $\mathbb{T}$ can be modeled via the measure preserving system $(\T^2, m_{\text{Leb}}, S)$, where $S: \T^2 \to \T^2$ is given by
\[
S(x,y):=(x+\alpha, y+x), \quad x,y \in \T.
\]
Indeed, given $f \in C(\T),$ in order to study the statistical properties of the sequence $a(n):=f(n^2\alpha),$ we have
\[
\lim_{N \to \infty} \frac{1}{N}\sum_{n=1}^N \prod_{j=1}^{r} a(n+h_i) = \int_{\T^2} \prod_{j=1}^{r} g(S^{h_i}(x,y)) \ dm_{\text{Leb}}(x,y) ,
\]
where $g(x,y)=f(y)$, and the equality above holds for all $r \in \N$ and all $h_1,\dots,h_{\ell} \in \Z$.
This is an example of what we will call a \emph{Furstenberg system} (see Definition~\ref{D: FS} below). The study of such systems is especially important and has received a lot of attention recently in view of the existing connection with problems in analytic number theory. More precisely, one can see Chowla's conjecture as the claim that the Liouville function $\lambda$ can be modelled with a Bernoulli system, or that the M\"obius function $\mu$ can be modeled by a direct product of a procyclic system and a Bernoulli system (see \cite{EKLR17, Sar11, Sar12}). What is known currently are some partial results; we refer the reader to \cite{FH18, FH21, GLR21}.

There is also recent work on the topic of Furstenberg systems, showcasing the importance of this topic in recent years (see for example \cite{FLR25, KKLR23}). These show the attention that Furstenberg systems have received, highlighting their relevance in the development of ergodic theory, especially taking into account, as we mentioned, its very important and strong connections with number theory (see also \cite{L23} for its connection with Sarnak's conjecture),  although they are of interest for its own sake as well.

Expanding on the aforementioned connection, Frantzikinakis \cite{Fra22}, the work of whom is the main source of inspiration for the present article, dealt with iterates that come from smooth functions in a Hardy field (in fact, in the examples under consideration, the functions are all of the $C^{\infty}(\R_+)$ class). In particular \cite[Theorem~1.1]{Fra22} determines the structure of Furstenberg systems with respect to sequences of polynomial growth rate.\footnote{ By this we mean that the sequence grows slower than some $(n^k)_n,$ $k\in\N$.}

In this article, we will go in a slightly different but also very interesting (and less explored) direction: we will look into iterates of the form $(n^{\log^c n})_n$ and $(\alpha\lfloor n^{\log^c n}\rfloor)_n,$ for $0<c<1/2$ and $\alpha\in\R\setminus\Z.$ As is easy to check, the sequence $(n^{\log^c n}=\exp(\log^{c+1}n))_n$ grows faster than any polynomial, but it is of subexponential growth because of the logarithm. In contrast with Frantzikinakis's results, the systems we obtain are all Bernoulli (Theorem~\ref{mainthm}). This latter result follows through the study of exponential sums and a strong equidistribution result on a multidimensional torus of shifts of the original sequence (Theorem~\ref{SPexpsum}) which also helps us derive some corollaries with a flavor of topological dynamics (Corollary~\ref{topcor intro}), and a norm convergence result for systems with commuting transformations (Theorem~\ref{T: vN}). 

Dealing with sequences of intermediate growth has remained a widely open topic (cf. the open problems in \cite{Fr16} and \cite{Fr21} for such sequences of iterates) precisely because traditional approaches and many crucial tools that one usually uses to show more general convergence (and consequently recurrence results), such as the ubiquitous van der Corput trick, can no longer be used. Hence, one has to consider alternative approaches, e.g., \cite[Theorem~1.10 (i)]{Fr21}, where Frantzikinakis showed that for distinct $c_1,\ldots, c_r\in (0,1/2)$ the sequences $(\floor{n^{\log^{c_1}n}})_n, \ldots, (\floor{n^{\log^{c_r}n}})_n$, are jointly ergodic for every ergodic nilsystem. He achieved this via the main result on joint ergodicity of sequences (\cite[Theorem 1.1]{Fr21}) which allows one to prove the property of joint ergodicity by showing that said sequences enjoy the equidistribution and seminorm estimates properties. Because in his setup, the underlying system is a nilsystem, the seminorm estimates trivally hold, and because of the fact that the $c_i$'s are distinct, he can trivially apply \cite[Theorem~2]{bp} to deduce the required equidistribution.

On the other hand, one should be careful when aiming for more general results. As it turns out, the choice and properties of the sequence of intermediate growth in this paper is important, as \cite[Theorem 1.6]{Boshernitzan-equidistribution} shows that even equidistribution might fail when one considers general Hardy field iterates of intermediate growth.


\subsection{Notation}
We will often make use of the following notation:
For $N \in \N$ we let $[N]:=\{1,\dots,N\}$ and given a bounded sequence $a: \N \to \C$ we write, for any bounded non-empty subset $A$ of $\N$
\[
\E_{n \in A} a(n):=\frac{1}{|A|} \sum_{n \in A} a(n).
\]
Similar notation is used for sequences of measures. Next, given $a, b : \R_+ \to \R$ we write
$a(t) \ll b(t)$ if there exists some absolute constant $C>0$ such that $|a(t)| \leq C |b(t)|$ for all large enough $t \in \R$. For a real number $t,$ let $\lfloor t\rfloor$ denote the integral part of $t,$ that is, the greatest
integer less than $t$; let $\{t\} = t - \lfloor t\rfloor$ be the fractional part of $t.$ Finally, for $t \in \R$ we put $e(t):=e^{2\pi it}$.

\subsection{Main results and applications}
We now state the main results that we obtain for intermediate growth sequences. 

 \emph{Throughout the rest of the section, for $0<c<1/2,$ we let $G(x):=x^{\log^c x}.$}

First, we get a characterization of associated Furstenberg systems in the following theorem. We note that in this paper, we will say a measure preserving system is Bernoulli if it is of the form $(X^{\Z},\otimes_{n \in \Z} \mu, S),$ where $(X,\mu)$ is a probability space, and $S: X^\Z \to X^\Z$ is the shift given by $S(x_n)_{n \in \Z}:=(x_{n+1})_{n \in \Z}$ (in particular, the trivial one-point Bernoulli system is allowed).
\begin{theorem}\label{mainthm}
Let $0<c<1/2$. Then, the sequences
 $a(n):=e(G(n))$ and $b(n):=e(\alpha \lfloor G(n) \rfloor),$ $\alpha \in \R \setminus \Z,$
have unique Bernoulli Furstenberg systems.
\end{theorem}

Because of the special disjointness properties that Bernoulli systems enjoy, Theorem~\ref{mainthm} allows us to obtain the following corollary that has a flavor of topological dynamics. We will use $h(T)$ to denote the topological entropy of a topological dynamical system $(X,T)$.

\begin{corollary}\label{topcor intro}
Let $(X,T)$ be a topological dynamical system with $h(T)=0$, $0<c<1/2$. Then, for the sequences $a(n):=e(G(n))$ and $b(n):=e(\alpha\lfloor G(n)\rfloor),$ $\alpha\in\mathbb{R}\setminus \mathbb{Z},$  for all $x \in X$ we have
\begin{equation}\label{topcor}
    \lim_{N \to \infty} \E_{n \in [N]} a(n)f(T^nx)=0\;\;\text{and}\;\;\lim_{N \to \infty} \E_{n \in [N]} b(n)f(T^nx)=0.
\end{equation}
\end{corollary}



The following equidistribution result is a crucial ingredient for the proof of Theorem~\ref{mainthm} and improves on results of \cite{bp}.

\begin{theorem}\label{T: equi}
Let $0<c<1/2$, distinct $h_1,\dots,h_r \in \Z$, and $\alpha\in \R\setminus \Z$. Then the sequences \[((G(n+h_1),\ldots,G(n+h_r)))_n\;\;\text{and}\;\;((\alpha\lfloor G(n+h_1)\rfloor,\ldots, \alpha\lfloor G(n+h_r)\rfloor))_n\] are equidistributed in $\T^r$ and $X^r$ respectively, where $X=\overline{\{\alpha n \pmod 1 : n \in \Z\}}$, equipped with its Haar measure.\footnote{ Which is equal to $m_{\text{Leb}}$ in $X=\T$ when $\alpha\in \R\setminus\Q,$ and the measure in $X=\{0,1/q,\ldots, (q-1)/q\}$ which assigns $1/q$ to the values $j/q,$ $0\leq j\leq q-1,$  when $\alpha=p/q\in \Q\setminus \Z,$ $(p,q)=1$.} 
\end{theorem}

It is worth noting here that for ``smooth enough'' (i.e., polynomial \cite{W16} or Hardy field \cite{Boshernitzan-equidistribution}) functions of polynomial growth, we usually have equidistribution results for the corresponding sequences when they are of different growth rates or they satisfy strong independence conditions. In our case though that we have intermediate growth,\footnote{ By this we mean that $G$ grows faster than any polynomial but slower than any exponential function; fact that it is easy (via L'H\^opital's rule) to also verify for every derivative of $G.$} we see that it is enough to take shifts of the function $G(x).$ Actually, an even stronger equidistribution result holds for the distinct shifts of $G$ (for the definition of ``goodness'' check Definition~\ref{D: good} below).

\begin{theorem}
\label{SPexpsum}
Let $0<c<1/2$ and distinct $h_1,\dots,h_r \in \Z$. Then the sequence \[((G(n+h_1),\ldots,G(n+h_r)))_n\] is good. 
\end{theorem}

Because of the latter results, we have the following von Neumann-type result for a probability system $(X,\mathcal{B},\mu)$ with commuting measure preserving transformations $T_1,\dots, T_r$.

\begin{theorem}\label{T: vN}
For $r\in \mathbb{N}$ let $(X,\mathcal{B},\mu,T_1,\ldots,T_r)$ be a measure preserving system and $f\in L^2(\mu)$ and let $0<c<1/2$. Then, for distinct $h_1,\ldots, h_r\in \mathbb{Z}$ we have  \[\lim_{N\to\infty}\norm{\E_{n \in [N]} T_1^{\left\lfloor G(n+h_1)\right\rfloor}\cdots T_r^{\left\lfloor G(n+h_r)\right\rfloor}f-Pf}_2=0,\] where $P$ denotes the projection on the set $\{f\in L^2(\mu):\;T_i f=f$ for all $1\leq i\leq r\}.$
\end{theorem}

\begin{remark}
In the special case where $r=1,$ the previous result follows by the spectral theorem together with \cite[Theorem~2]{bp}.

\end{remark}





\subsection{Structure of the paper}
The structure of the paper is as follows. In section~\ref{S:2}, we give the necessary background that is required to follow the rest of the paper: some basics on measure preserving systems, the relevant definitions for one of the main objects of study: Furstenberg systems, and also some necessary basic results on Hilbert space splittings and equidistribution of real-valued sequences on tori.

Section~\ref{Sec:proofs} is the bulk of the paper and concerns the proofs of some intermediate results on bounds of derivatives of the functions of interest. These are based on the ideas 
of the proof of Theorem 2 of \cite{bp}, but they have to be modified and adapted so that we can actually deal with the linear combinations that naturally appear when we start considering the averages that determine a Furstenberg system.

The aforementioned results help us prove, in Section~\ref{sec: proofs of main}, the main exponential sum theorem (Theorem~\ref{SPexpsum}) which in turn implies the equidistribution properties of our sequences (Theorem~\ref{T: equi}). Once these results are established, we can use them to characterize the Furstenberg systems associated to the corresponding sequences with intermediate growth (Theorem~\ref{mainthm}), some corollaries with topological dynamics flavor (Corollary~\ref{topcor intro}), as well as a von Neumann type ($L^2$) convergence result (Theorem~\ref{T: vN}).

\subsection{Open problems}
It is natural to consider, as is done in \cite{Fra22}, the study of Furstenberg systems associated to sequences of the form $d(n):=f(T^{\lfloor G(n) \rfloor}x)$, where $x \in X$ is a ``typical'' point of a measure preserving system $(X,\mathcal{B},\mu,T)$, $f \in L^{\infty}(\mu)$ and $0<c<1/2$.

In order to be able to obtain a result similar to \cite[Theorem~1.6]{Fra22} or \cite[Corollary~1.7]{Fra22} one would need to have seminorm control for averages of the form
\[
\lim_{N \to \infty} \E_{n \in [N]} f(T^{\lfloor G(n+h_1)\rfloor}x)\cdots f(T^{\lfloor G(n+h_r)\rfloor}x).
\]
This is currently beyond our reach, as the sequence $(\lfloor G(n)\rfloor)_n$ grows too fast for the van der Corput trick to be helpful, and we cannot even attempt to prove a strong stationarity type of property for them by reducing the problem to nilsystems. 

However, given the independence that shifts of the function $G(x)$ enjoy, and the strong equidistribution results we were able to prove, it is natural to expect the result to be true, so we leave it as a conjecture.

\begin{conjecture}\label{conjecture: pointwise}
Let $(X,\mathcal{B},\mu,T)$ be a measure preserving system, $0<c<1/2$ and $f \in L^{\infty}(\mu)$. Then $\mu$-a.e. $x \in X$ is such that the Furstenberg system associated to $d(n):=f(T^{\lfloor G(n)\rfloor}x)$ is unique and Bernoulli. 
\end{conjecture}

As an intermediate step, one could try to prove the special case of Conjecture~\ref{conjecture: pointwise} when $X=G/\Gamma$ is a nilmanifold equipped with its Borel $\sigma$-algebra, $T=R_a$ is the nilrotation $R_a(g\Gamma)=ag\Gamma$, and $\mu$ is the Haar measure on $X$. In this case, one would consider the case where $f \in C(X)$. We also note that this is open even for nilrotations on the torus (for example, the case where $X=\T^2$ and $T(x,y):=(x+\alpha,y+x)$ for some $\alpha \in \R \setminus \Q$).

If one were able to prove Conjecture~\ref{conjecture: pointwise}, then, it is also likely that the following result would hold.
\begin{conjecture}\label{conjecture: l2convergence}
Let $(X,\mathcal{B},\mu)$ be a probability space, and $T, S: X \to X$ two (not necessarily commuting) measure preserving transformations. Let $0<c<1/2$ and suppose that $(X,\mu,T)$ has zero entropy. Then,
\begin{itemize}
    \item For every $f, g \in L^{\infty}(\mu)$ we have
     \[
    \lim_{N \to \infty}\norm{\E_{n \in [N]} T^nf S^{\lfloor G(n)\rfloor}g -\E[f|\mathcal{I}_T]\E[g|\mathcal{I_S}]}_2=0,\footnote{ $\E[\cdot|\mathcal{I}_T]$ and $\E[\cdot|\mathcal{I}_S]$ denote the conditional expectations on $\mathcal{I}_T$ and $\mathcal{I}_S$, the $\sigma$-algebras of $T$-invariant and $S$-invariant sets of $X$ respectively.}
    \]
    where the limit is taken in $L^2(\mu)$, and 
    \item For every $A \in \mathcal{B}$ it is
    \[
    \lim_{N \to \infty} \E_{n \in [N]} \mu( A \cap T^{-n}A \cap T^{-\lfloor G(n) \rfloor}A) \geq \mu(A)^3.
    \]
\end{itemize}
\end{conjecture}
We would like to remark that in the commutative case, i.e., $TS=ST$, we expect Conjecture~\ref{conjecture: l2convergence} above to be true without the assumption that $(X,\mu,T)$ has zero entropy.

\medskip

\noindent {\bf{Acknowledgements.}} We wish to thank Nikos Frantzikinakis who not only suggested the sequences that we dealt with, but also made comments on the first draft of the article which led to its improvement. 

\section{Background}\label{S:2} 
\subsection{Measure preserving systems}
Let $(X,\mathcal{B},\mu)$ be a probability space. Let $T_1,\dots,T_r: X \to X$ be commuting invertible measure preserving transformations. We will say that the tuple $(X,\mathcal{B},\mu,T_1,\dots,T_r)$ is a measure preserving system. Oftentimes we will only use single transformation measure preserving systems, i.e., we will have $r=1$. Recall that a system is ergodic if the only sets invariant under the action generated by $\langle T_1,\dots, T_r\rangle$ have either full or zero measure. 

For other common concepts in ergodic theory, such as conditional expectations, factor, Bernoulli system, joining, entropy, and disjoint systems we refer the reader to Section 2 of \cite{HK18}.

\subsection{Furstenberg systems}
The purpose of this subsection is to present some basic facts about one of our main objects of study in this paper: Furstenberg systems. We will follow the presentation in \cite{Fra22} adapted to our slightly simpler context.
\begin{definition}
Let $a: \Z \to \S^1$ be a sequence. We say that $a$ \emph{admits correlations} if there exists a subsequence of increasing natural numbers $(N_k)_k$ such that the limits
\[
\lim_{k \to \infty} \E_{n \in [N_k]} a^{\varepsilon_1}(n+h_1)\cdots a^{\varepsilon_r}(n+h_r)
\]
exist for all $r \in \N$, all distinct $h_1,\dots,h_r \in \Z$ and all $(\varepsilon_1,\dots,\varepsilon_r) \in \Z^r \setminus \{\vec{0}\}$.
\end{definition}

Even though Furstenberg systems can be defined over different sequences of intervals, we will only be interested in those arising from the standard intervals $[1,N]$ and subsequences of those.

\begin{remark}
It is worth pointing out that if $a$ admits correlations, then, by letting $S: (\S^1)^\Z \to (\S^1)^{\Z}$ be the shift map $S(y_n)_{n \in \Z}:=(y_{n+1})_{n \in \Z}$ one can show that the weak* limit $\lim_{k \to \infty} \E_{n \in [N_k]} \delta_{S^n a}$ exists.    
\end{remark}
We are now in position to give the definition of Furstenberg system.

\begin{definition}[Furstenberg system of a sequence]\label{D: FS}
Let $a: \Z \to \S^1$ be a sequence that admits correlations on $([N_k])_{k \in \N}$. Put $X:=(\S^1)^{\Z}$, equipped with the product topology, $\mathcal{B}:=\text{Borel}(X)$, and $S: X \to X$ be the shift map as defined above. Let $\mu:=\lim_{k \to \infty} \E_{n \in [N_k]} \delta_{T^n a}$ be a weak* limit point of the sequence of measures $\left( \E_{n \in [N]} \delta_{T^n a}\right)_{n}$. We say that the system $(X,\mathcal{B},\mu,T)$ is a \emph{Furstenberg system} of $a$ on $([N_k])_{k}$.

We say that the Furstenberg system of $a$ is \emph{unique} if $a$ admits correlations on the traditional intervals $([N])_{N \in \N}$.
\end{definition}
\begin{remarks}
(1) If we are given a sequence $a: \N \to \S^1$, we simply extend it in any way to a sequence $a: \Z \to \S^1$; the measure $\mu$ is not affected at all by this extension.
\\ \\
\noindent(2) It is shown in \cite{bf21} that under some natural restrictions, such Furstenberg systems are unique, in the sense that they do not depend on the construction of the measure space, but are rather characterized by the statistical properties of the sequences.
\end{remarks}
We will be especially interested in determining when the measure $\mu$ is the product measure $\bigotimes_{n \in \Z} m_{\text{Leb}}$. To this effect, we have the following result, which is a consequence of the Stone-Weierstrass theorem. 
\begin{proposition}\label{Fsystemchar}
Let $a: \N \to \S^1$ be a sequence and $(X,\mathcal{B},\mu,T)$ be a Furstenberg system for $a$. Let $F: X \to \S^1$ be given by $F(x):=x(0).$ Then, the measure $\mu$ is determined by the values:
\[
\int_X T^{h_1}F^{\varepsilon_1}\cdots T^{h_r} F^{\varepsilon_r} \ d\mu,
\]
for all $r \in \N$, all distinct $h_1,\dots, h_r \in \Z$, and all $(\varepsilon_1,\dots,\varepsilon_r) \in \Z^r$.
\end{proposition}
\begin{proof}
The polynomials $z^n: \S^1\to \S^1, n \in \Z$ generate a $C^*$ algebra of $C(\S^1)$ that separates points. Therefore, with our notation above, we have that the family
\[
\mathcal{A}:=\{ T^{h_1}F^{\varepsilon_1}\cdots T^{h_r}F^{\varepsilon_r}: r \in \N, \text{ distinct } h_1,\dots,h_r \in \Z, (\varepsilon_1,\dots,\varepsilon_r) \in \Z^r\}
\]
generates a $C^*$-algebra that separates points in $X$. 

Now, by the Stone-Weierstrass theorem, the $C^*$-algebra generated by these functions is exactly equal to $C(X)$, and by the Riesz-Fischer theorem, the measure $\mu$ is determined by its values on $C(X)$. Thus, by linearity of $\mu$, knowing its values on $\mathcal{A}.$ 
\end{proof}

Proposition~\ref{Fsystemchar} allows us to deduce a convenient characterization of sequences whose Furstenberg system is unique and Bernoulli.

\begin{corollary}\label{cor: FsysBern}
Let $a: \Z \to \S^1$ be a sequence such that 
\begin{equation}\label{corFsyshyp}
\lim_{N \to \infty} \E_{n \in [N]} a^{\varepsilon_1}(n+h_1)\cdots a^{\varepsilon_r}(n+h_r)=0,
\end{equation}
for all $r \in \N$, all distinct $h_1,\dots,h_r \in \Z$ and all $(\varepsilon_1,\dots,\varepsilon_r) \in \Z^r\setminus\{\vec{0}\}$. Then, $(a(n))_{n\in\Z}$ has a unique Furstenberg system that is Bernoulli (with $\mu=\otimes_{n \in \Z} m_{\text{Leb}}$).
\end{corollary}

\begin{proof}
Observe that for all $k \in \Z \setminus \{0\}$ we have $\int_{\S^1} z^k \ dm_{\text{Leb}}(z)=0$. With the notation as in Proposition~\ref{Fsystemchar} above, it is clear that the Bernoulli measure on $(\S^1)^{\Z}$ satisfies
\[
\int_X T^{h_1}F^{\varepsilon_1}\cdots T^{h_r} F^{\varepsilon_r} \ d\otimes_{n \in \Z} m_{\text{Leb}} = 0,
\]
provided the vector $(\varepsilon_1,\dots,\varepsilon_r) \neq \vec{0}$. That the equality above holds for $\mu$ is a trivial consequence of \eqref{corFsyshyp}, so by Proposition~\ref{Fsystemchar}, the two must coincide.
\end{proof}

\subsection{Hilbert space splittings}

We first recall a Hilbert space splitting theorem and a version of the classical Bochner-Herglotz theorem which we will use for the space $L^2(\mu).$

\begin{theorem}[\cite{Ber96}]\label{T:split}
For $r\in \mathbb{N}$ let $U_1,\ldots,U_r$ be commuting unitary operators on a Hilbert space $(\mathcal{H},\norm{\cdot}).$ If 
\[\mathcal{H}_{\text{inv}}:=\{f\in \mathcal{H}:\;U_i f=f\;\text{for all}\;1\leq i\leq r\},\] and 
\[\mathcal{H}_{\text{erg}}:=\left\{f\in \mathcal{H}:\;\lim_{N_1,\ldots,N_r\to\infty}\norm{\frac{1}{N_1\cdots N_r}\sum_{n_1=1}^{N_1}\cdots\sum_{n_r=1}^{N_r}U_1^{n_1}\cdots U_r^{n_r}f}=0\right\},\] then
\[\mathcal{H}=\mathcal{H}_{\text{inv}}\oplus \mathcal{H}_{\text{erg}}.\]
\end{theorem}

\begin{theorem}\label{T:BH}
For $r\in \mathbb{N}$, let $U_1,\ldots,U_r$ be commuting unitary operators on a Hilbert space $\mathcal{H}$ and $f\in \mathcal{H}.$ Then, there exists a measure $\nu_f$ on $\mathbb{T}^r$ such that
\[\langle U_1^{n_1}\cdots U_r^{n_r}f,f\rangle=\int_{\mathbb{T}^r}e(\alpha_1 n_1+\cdots+ \alpha_r n_r)\;d\nu_f(\alpha_1,\ldots,\alpha_r),\] 
for all $(n_1,\ldots,n_r)\in \mathbb{Z}^r.$
\end{theorem}

A standard consequence that can be easily derived from Theorems~\ref{T:split} and \ref{T:BH} is the following.

\begin{corollary}\label{spectralET}
Let $(X,\mathcal{B},\mu,T_1,\dots,T_r)$ be a measure preserving system. Let $f \in L^2(\mu)$. Then, if we let $P$ denote the projection onto $L^2_{\text{inv}}$ it is
\[
\nu_f(\{(0,\dots,0)\})= \langle Pf, f\rangle.
\]
\end{corollary}

\subsection{Equidistribution} We recall now some basic definitions of equidistribution that we shall use in the sequel.



\begin{definition}
We say that a sequence $(x_n=(x_n^1,\ldots, x_n^r))_n \subseteq \R^r$ is \emph{equidistributed} on $\T^r$ if, for all Riemann integrable coordinate-wise $1$-periodic functions $f$ we have
\[
\lim_{N \to \infty} \E_{n \in [N]} f(\{x_n\}) = \int_{\T^r} f(x) \ dm_{\text{Leb}}(x),
\] where $\{x_n\}=(\{x_n^1\},\ldots, \{x_n^r\}).$
\end{definition}

We state next a very useful result of Weyl that allows one to easily check equidistribution of sequences.

\begin{theorem}[Weyl's criterion for equidistribution, \cite{W16}]\label{T: Weyl}
A sequence $(x_n)_n \subseteq \R^r$ is equidistributed on $\T^r$ if and only if for every $\vec{k} \in \Z^r \setminus \{\vec{0}\}$ we have
\[
\lim_{N \to \infty} \E_{n \in [N]} e(\vec{k}  \cdot x_n) = 0.
\]
\end{theorem}

We also recall a stronger equidistribution notion from \cite{Fr16}.\footnote{ The notion was defined in \cite{Fr16} for sequences of real variable polynomials.}

\begin{definition}\label{D: good}
We say that a sequence $(x_n)_n \subseteq \R$ is \emph{good} if for every $\alpha\in\R\setminus\{0\}$ we have
\[
\lim_{N \to \infty} \E_{n \in [N]} e(\alpha\cdot x_n)=0.
\] 
A sequence $(x_n=(x_n^1,\ldots, x_n^r))_n \subseteq \R^r$ is \emph{good} if every non-trivial linear combination of the
sequences $(x_n^1)_n, \ldots, (x_n^r)_n$ is good.
\end{definition}

We remark that via Theorem~\ref{T: Weyl} it is immediate that every good sequence is equidistributed. As an example, the sequence $(\sqrt{17}n)_n$ is equidistributed but not good, which means that the latter notion is a strictly stronger than the former.

Finally, we need an equidistribution notion for certain finite abelian groups.

\begin{definition}
Let $m \in \N$ with $m \geq 2$, and consider the set $X=\{0,1/m,2/m,\dots,(m-1)/m\}$. We say that a sequence $(x_n)_n \subseteq X$ is \emph{equidistributed} on $X$ if for every function $f$ on $X$, it is
\[
\lim_{N \to \infty} \E_{n \in [N]} f(x_n) = \int_X f \ d\mu,
\]
where $\mu$ is the Haar measure on $X$.
\end{definition}





\section{Bounds for the functions of interest}\label{Sec:proofs}
In this section we provide all the ingredients towards the exponential sum results that are the backbone of our main results.
Once again, $G(x)=x^{\log^c x},$ $0<c< 1/2,$ throughout the section. 

 We consider the sequence $a(n):=e(G(n))$ in $\S^1$. It is very important to remark that in all the statements proven below, we first pick the shifts $h_j$ and the accompanying coefficients $\alpha_j$.


In order to prove Theorem~\ref{SPexpsum}, we will follow \cite{bp} and adapt it to our context.

Given $r \in \N$, distinct $h_1,\dots,h_r \in \Z$, and $(\alpha_1,\dots,\alpha_r) \in \R^r \setminus \{\vec{0}\}$, we set
\[F(x)=F_{h_1,\dots,h_r,\alpha_1,\dots,\alpha_r,c}(x):=\sum_{j=1}^r \alpha_j\cdot G(x+h_j);\]
a piece of notation that will be used throughout the section.


The following theorem, the proof of which is at the end of the section, gives us the equidistribution of the seqeunce $((G(n+h_1),\ldots, G(n+h_r)))_n$ of Theorem~\ref{T: equi}.

\begin{theorem}\label{thm2aif98}
Let $0<c<1/2$. Given $0<\gamma<1$ and $0<C<c$, there exists a constant $\kappa>0$ such that uniformly in $G(2N)^{-\gamma} \leq |\alpha| \leq G(N)^C$, for $N$ big enough one has
\[
\sum_{N+a \leq n \leq 2N-a} e(\alpha F(n)) \ll N\exp\left(-\kappa\log^{1-2c} N\right),
\]
where $a=\max\{ |h_i|: 1 \leq i \leq r\}$.
\end{theorem}
The main tool we will use to prove Theorem~\ref{thm2aif98} is the following result due to Karacuba (\cite[Theorem 1]{Ka71}), which we restate in the more convenient form given in \cite{bp} below.
\begin{lemma}[Lemma 10, \cite{bp}]
\label{lem10aif98}
Let $\theta, \Theta,\Theta_0, \delta$ be real numbers satisfying the inequalities
\[
0<\delta<1 \quad \text{and} \quad 0<\Theta<\theta<\Theta_0<1,
\]
and let $\mathcal{F}=\mathcal{F}(N,\theta,\Theta,\Theta_0,\delta)$ be the set of all functions $f(x)$ for which there exists $k \geq 3$ such that $f$ is $k$-times differentiable on $[N,2N]$,
\[
\left| \frac{f^{(k)}(x)}{k!} \right| \leq N^{-\Theta_0k} \quad(N < x \leq 2N)
\]
and there exists a set $S \subseteq \{1,2,\dots,k-1\}$ with cardinality at least $\delta k$ and
\[
N^{-\theta s} \leq \left| \frac{f^{(k)}(x)}{k!} \right| \leq N^{-\Theta s} \quad (N < x \leq 2N)
\]
for all $s \in S$. Then, there exist constants $B, c_1>0$ depending only on $\theta,\Theta,\Theta_0$, and $\delta$, such that for any $f \in \mathcal{F}$ one has
\[
\left| \sum_{N \leq n \leq 2N} e(f(n))\right| \leq BN^{1-\frac{c_1}{k^2}}.
\]
\end{lemma}

Therefore, our goal will be to find suitable $\theta, \Theta, \Theta_0, \delta$ that exhibit that our function $F(x)$ is in $\mathcal{F}$ (regardless of the choice of $0<c<1/2$).

The following two lemmas, Lemma~14 and Lemma~15 from \cite{bp}, that we will repeatedly use, give explicit bounds on the derivatives of $G(x)$. We restate them here for the convenience of the reader.

\begin{lemma}[Lemma~14, \cite{bp}]\label{lem14aif98} Let $0<c<1/2$. Then, there is a constant $x_0=x_0(c)$ such that for all $s \in \N$ and all $x \geq x_0$ with $(c+1) \log^{c} x\geq s$ one has
\[
\frac{G^{(s)}(x)}{s!} \geq\frac{G(x)}{2x^s}.
\]
\end{lemma}

\begin{lemma}[Lemma~15, \cite{bp}]\label{lem15aif98} Let $c>0$ and $N\geq 10$. Then, for all $s \in \N$ and $N<x\leq 2N$ we have
\[
\left| \frac{G^{(s)}(x)}{s!} \right| \leq 2\left(\frac{2}{N}\right)^sG(3N).
\]
\end{lemma}

In the next two subsections we will show similar bounds to the previous two lemmas but for $F^{(s)}(x)$ in place of $G^{(s)}(x).$

\subsection{Lower bound}
Our first main goal will be the proof of a lower bound for the derivatives of $F(x)$. We shall accomplish that with the following proposition.


\begin{proposition}\label{lowerbound}
Let $0<c<1/2$ and $N \in \N$ be large enough (depending on the $h_1,\ldots, h_r, \alpha_1,\ldots, \alpha_r$, and $c$). Then, there exist $x_0, d>0$ and $\tau\in \mathbb{N}\cup\{0\}$ all depending only on $h_1,\ldots,h_r,\alpha_1,\dots,\alpha_r, c$ such that for all $x \geq x_0$ with $(c+1)\log^{c}(x+h_i) \geq s$, $i=1,\dots, r$ it is 
\[
\frac{F^{(s)}(x)}{s!} \geq \frac{d\cdot G(x)}{x^{s+\tau}} \quad (N+a < x \leq 2N-a).
\]
\end{proposition}

\begin{remark}
It is important to note that in this result and the next ones, we will keep assuming that we can work with $-F$ if necessary for all the limits to be positive (also allowing for $+\infty$). This is without loss of generality for $F$ (and in turn all its derivatives) because we can take conjugates in the exponential sums we want to estimate.
Moreover, because of Lemma~\ref{lem14aif98} and Lemma~\ref{lem15aif98}, we have that all derivatives of $G(x)$ are increasing to $+\infty$ and the higher the degree of a derivative is the slower it grows (see also the computations in the proof of Proposition~\ref{lowerbound} below). We will use these facts repeatedly.
\end{remark}

Before giving a proof of Proposition~\ref{lowerbound} we will look at a few examples to understand what kind of behavior the linear combinations of $G(x)$ that define $F(x)$ can have. This will help clarify some of the key ideas in the proof of Proposition~\ref{lowerbound}, and also understand some of the difficulties that can arise. 
\subsubsection{Examples}
Although it will perhaps not be fully apparent in the proof, the driving force behind the computations is, first and foremost the sum of the integer coefficients. If this sum is not $0$, regardless of the shifts (assuming, as we can, that this sum of coefficients $\Sigma$ is bigger than $0$), then we have the following trivial bounds afforded by the mean value theorem.
\begin{example}
Suppose we have a function $F(x)$ of the form
\[
F(x):=G(x+2)-3G(x+1)+3G(x).
\]
Then, clearly $G(x+2) \geq G(x)$ since $G$ is increasing. Also, $G(x+1) \leq 1.1 G(x)$ for $x \geq x_0,$ some objective $x_0$ (it only really depends on our ``choice'' of $1.1$, and that the combination is $G(x+1)$ and $G(x)$). Of course, the choice of $1.1$ and of $x_0$ will also depend on $c$. The choice of $1.1$ is, of course, arbitrary; it simply is good enough for this example, but it admits a suitable generalization for the more general cases.

Combining all our previous comments and realizations, it is
\begin{eqnarray*}
F(x) \geq  G(x)-3.3G(x)+3G(x)= 0.7G(x)\geq \frac{1}{2}G(x).
\end{eqnarray*}
This would be a good enough lower bound, even though with some care one can improve on the $1/2$ constant and get it as close to $1$ as needed, though, as we saw, this is not necessarily helpful.
\end{example}
The main reason we cannot use this argument right away in the proof of Proposition~\ref{lowerbound}, is that we would need to replace $F(x)$ by $F^{(s)}(x)$, but we would not have the required uniformity just from using general properties of limits of quotients of $G(x)$ together with L'H\^opital's rule.

Next, we illustrate the sorts of different behaviors that arise immediately when the sum of the coefficients is null ($\Sigma=0$). We first show an example of coefficients we will label as ``balanced''.
\begin{example} Consider the following expression for a ``balanced'' $F(x)$:
\[
F(x):=G(x+2)-2G(x+1)+G(x).
\]
Then, it is easy to see that
\begin{eqnarray*}
F(x)  =  \int_{x}^{x+1}\int_t^{t+1} G''(s)  \ ds \ dt\geq  \int_x^{x+1} G''(t) \ dt \geq  G''(x),
\end{eqnarray*}
given that any derivative of $G$ is increasing, as we saw before.
\end{example}
This example showcases how the coefficients can ``conspire'' to force $F(x)$ to have slower growth. In fact, this argument could be applied straight away for the ``balanced'' cases of Proposition~\ref{lowerbound}. The issue is actually being able to identify when the coefficients are or are not such that they admit an integral representation like the one above (more on this in an example below).

On the other hand, it is possible for the sum of coefficients to be $0$, and such that a single application of the mean value theorem is enough.
\begin{example}
Suppose we have $F(x):=-G(x+2)+3G(x+1)-2G(x)$. 

Then
\begin{eqnarray*}
F(x) &=& -G(x+2)+3G(x+1)-2G(x)\\
&=& -(G(x+2)-G(x+1))+2(G(x+1)-G(x))\\
&=&-G'(\xi_x)+2G'(\tau_x)\geq  2G'(x)-G'(x+2)\geq 0.9 G'(x),
\end{eqnarray*}
where we used the mean value theorem to find the $\xi_x$ and $\tau_x$, the monotonicity of $G',$ and the fact that $G'(x+2) \leq 1.1 G'(x)$ for $x$ large.

(we used the mean value theorem to find $\xi_x$ and $\tau_x$, and once again that $G(x)$ and its derivatives are all increasing).

Next, we would once again use that $G'(x+2) \leq 1.1 G'(x)$ for $x$ large, as before, to obtain the lower bound
$0.9 G'(x)$
which is good enough for our purposes.
\end{example}
Note that, once more, the issue with the generalization is the uniformity that we would need for these bounds for the derivatives of $G(x)$.

It is also important to notice that the distances in the jumps provided by the $h_i$ play an important role in deciding whether the resulting $F(x)$ will be balanced or not (this is going to be behind the issue that the method we mentioned above had).
\begin{example}
Assume that $F(x)$ is of the form $F(x):=G(x+4)-3G(x+1)+2G(x).$ 

We have
\begin{eqnarray*}
F(x) & = & G(x+4)-3G(x+1)+2G(x)\\
& = & G(x+4)-G(x+1) - 2(G(x+1)-G(x))\\
& \geq & 3G'(x+1)-2G'(x+1) = G'(x+1),
\end{eqnarray*}
once again by using the increasing properties of $G$ and its derivatives, after a simple application of the mean value theorem. Do note here, that the coefficients are the same as in the previous example, except we changed all the signs. Moreover, what determined who ``won'' in the battle of signs, depended both on the coefficients and the length between the shifts of the $G(x)$.
\end{example}
However, things can be very different even if they appear similar. Namely, the ``balanced'' behavior can also be achieved, even if the coefficients do not look, a priori, like $(-1)^k\binom{n}{k}$ (as we mentioned before, this makes the ``balanced'' case hard to identify). In this case, using the mean value theorem straight away will also not work.
\begin{example}
Consider the function $F(x):=2G(x+2)-3G(x)+G(x-4).$

Then
\begin{eqnarray*}
F(x) & = & 2G(x+2)-3G(x)+G(x-4)\\
&=& 2(G(x+2)-G(x)) -(G(x)-G(x-4))\\
&= & 2\int_x^{x+2} G'(t) \ dt - \int_{x-4}^x G'(t) \ dt\\
&\geq  &\int_{x-2}^{x} G'(t) \ dt - \int_{x-4}^{x-2} G'(t) \ dt + \int_{x-2}^x G'(t) \ dt - \int_{x-4}^{x-2} G'(t) \ dt\\
&=& 2\int_x^{x+2}\int_{t-4}^{t-2} G''(s) \ ds \ dt \geq 4 \int_x^{x+2} G''(t-4) \ dt \geq  8 G''(x-4) \geq 4 G''(x).
\end{eqnarray*}

This provides a lower bound of the form $\lambda G''(x)$, which is good enough, as we saw above. 
\end{example}
It is also important to notice that the arguments that deal with ``balanced'' $F(x)$ do extend effortlessly to all its derivatives because we can write the combination of $F(x)$ as a single integral, so given the regularity of $G(x)$, we could do the same with all its derivatives.

In order to prove Proposition~\ref{lowerbound}, overcoming the difficulties mentioned above, we begin by first showing the following result, which can be interpreted as a suitable change of basis.

\begin{lemma}\label{changeofbasis}
Let $F$ be as defined above and let $s \in \N \cup \{0\}$. Let $\Delta$ denote the finite difference operator, i.e., $\Delta H(x):=H(x+1)-H(x)$ for a given function $H: \Z \to \C$. Then, there exists some polynomial $p \in \R[x]$ such that 
\[
F^{(s)}(x) = p(\Delta)G^{(s)}(x).
\]
\end{lemma}
\begin{proof}
We first begin with a reduction: without loss of generality we will assume that all the $h_i \geq 0$. We can do this by shifting the sum by the smallest of the $h_i$ that is negative. This does not change any result in the sums, because of the F\o lner property of the intevals $[1,N]$. 

Let us now suppose that $0 \leq a = \max \{ h_i: i=1,\dots,r\}$. Then, note that given a function $H: \Z \to \C$ we have (denoting by $I$ the identity operator)
\[
\begin{cases}
    IH(x)=H(x) \\
    \Delta H(x)=H(x+1)-H(x) \\
    \Delta^2 H(x)=H(x+2)-2H(x+1)+H(x) \\
    \vdots \\
    \Delta^aH(x)=\sum_{j=0}^a (-1)^j\binom{a}{j}H(x+a-j).
\end{cases}
\]
Applying these equations to $G(x)$ we obtain a way to write all of $G(x),G(x+1),\dots,$ $G(x+a)$ in terms of powers of the difference operator, already given in an upper echelon form. This means that we can solve the relevant system of equations backwards, allowing us to write
\[
F^{(s)}(x)=\sum_{i=0}^a d_i\Delta^aG^{(s)}(x),
\]
for some $d_i \in \R$, where we may assume that the non-zero coefficients are bounded away from $0$ with constants depending only on the $\alpha_j$ and $h_j$ appearing in the definition of $F$.
\end{proof}

We will also need some growth rate of the functions $\Delta^kG$, and a comparison between the growth rates of the derivatives of $G$ and shifts of it. We begin with the following lemma.
\begin{lemma}\label{simpleboundsDelta}
Let $s \in \N \cup \{0\}$, let $k \in \{0,\dots,a\}$, and let $G$ be as defined in Proposition~\ref{lowerbound} above. Then, 
\[
G^{(s+k)}(x+k) \geq\Delta^kG^{(s)}(x) \geq G^{(s+k)}(x) \quad (N+a<x\leq 2N-a).
\]
\end{lemma}
\begin{proof}
Indeed, observe that the following equality holds by induction:
\[
\Delta^kG^{(s)}(x) = \int_0^1\dots \int_0^1 G^{(s+k)}(x+u_1+\dots+u_k) \ du_1\dots du_k,
\]
so since $G^{(s+k+1)}(x)$ is positive on the indicated range by Lemma~\ref{lem14aif98}, it follows that $G^{(s+k)}$ is increasing, and thus
\[
G^{(s+k)}(x+k) \geq \Delta^kG^{(s)}(x) \geq G^{(k+s)}(x),
\]
as desired.
\end{proof}
We are now in position to give the proof of the lower bound we claimed in Proposition~\ref{lowerbound}.
\begin{proof}[Proof of Proposition~\ref{lowerbound}]
We begin by using Lemma~\ref{changeofbasis} so we have
\[
F^{(s)}(x)=\sum_{i=0}^a d_i\Delta^aG^{(s)}(x).
\]
Next, let $\tau = \min\{ i \in \{0,\dots, a\} : |d_i|>0\}$. Therefore, by Lemma~\ref{simpleboundsDelta} we can write
\[
F^{(s)}(x) \geq d_{\tau}G^{(s+\tau)}(x)-\lambda \sum_{j=\tau+1}^a G^{(s+j)}(x+a),
\]
where $\lambda=\max \{|d_j|: j \geq \tau+1\}$. Next write
\[
F^{(s)}(x) \geq d_{\tau}G^{(s+\tau)}(x)\left(1-\frac{\lambda}{d_{\tau}} \sum_{j=\tau+1}^a \frac{G^{(s+j)}(x+a)}{G^{(s+\tau)}(x)}\right).
\]
By Lemma~\ref{lem14aif98} and Lemma~\ref{lem15aif98}, we can give bounds for each of the summands ($j \geq \tau+1$):

\begin{eqnarray*}
\frac{G^{(s+j)}(x+a)}{G^{(s+\tau)}(x)} & \leq & \frac{(s+j)!}{(s+\tau)!} \cdot\frac{2^{s+j+1}G(3x+3a)}{\left(\frac{x}{2}\right)^{s+j}}\cdot\frac{2x^{s+\tau}}{G(x)} \\
&= &
4^{s+j+1}\frac{(s+j)!}{(s+\tau)!} \cdot\frac{1}{x^{j-\tau}}\cdot\exp(\log^{c+1}(3x+3a)-\log^{c+1}(x))\\
&\leq & 4^{(c+1)\log^{c}(x+a)+a+1} (s+a)^{a-\tau}\cdot\frac{1}{x}\cdot\exp\left((3x+3a)(c+1)\frac{\log^{c}(3x+3a)}{x}\right),
\end{eqnarray*}
by the mean value inequality. This last expression only depends on $x,c$ and $a$ (in other words, the values of the $h_i$). These are known a priori, which is crucial for the uniformity in $s$ that we need in the proofs later.

It is easy to check that the last expression tends to $0$ as $x \to \infty$, but again, we remark that this is the case depending only on the aforementioned parameters. This means that for a certain $x_0(c,h_1,\dots,h_r)$ it will be that this last quantity is bounded above by $\frac{d_{\tau}}{2\lambda}$, which implies that
\[
F^{(s)}(x) \geq \frac{d_{\tau}}{2}G^{(s+\tau)}(x).
\]
It follows that
\[
\frac{F^{(s)}(x)}{s!} \geq \frac{d_{\tau}}{2}\cdot \frac{G^{(s+\tau)}(x)}{s!} = \frac{d_{\tau}}{2}\cdot\frac{(s+\tau)!}{s!}\cdot \frac{G^{(s+\tau)}(x)}{(s+\tau)!} \geq \frac{d_{\tau}}{4}\cdot \frac{G(x)}{x^{s+\tau}},
\]
on the specified range, as we wished to show.
\end{proof}
In order to better understand the logic behind the proof of Proposition~\ref{lowerbound}, we refer the reader to a few illustrative examples which can be found in the Appendix.

\subsection{Upper bound and proof of Theorem~\ref{thm2aif98}} In this subsection, using Lemma~\ref{lem15aif98} and some elements of its proof, we will first obtain an upper bound for the absolute value of the derivatives of $F(x)$ (Lemma~\ref{upperbound}), which, together with Proposition~\ref{lowerbound}, implies Lemma~\ref{combinedineq}. The latter will help us prove Theorem~\ref{thm2aif98}; the remaining piece for the proof of Theorem~\ref{SPexpsum}, which will be presented in the next and final section.

We would like to highlight that in the argument in Lemma~\ref{upperbound} below, it is important that the functions $G$ and $F$ are both analytic. This allows us to use Cauchy's integral formula, which is very convenient to promptly give us the estimates we need.

\begin{lemma}\label{upperbound}
Let $c>0$ and $N$ be big enough (depending on the $h_i$ and the $\alpha_i$). Then, for all $s \in \N$ and $N<x \leq 2N$, where $a = \max\{ |h_i|: 1 \leq i \leq r\}$ we have
\[
\left| \frac{F^{(s)}(x)}{s!} \right| \leq C\left(\frac{2}{N}\right)^s G(3N),
\]
where $C$ is a positive constant depending on the $\alpha_i, h_i, 1\leq i\leq r.$ 
\end{lemma}

\begin{proof}
We follow the ideas of the proof of \cite[Lemma~15]{bp}. Let $F(z):=\sum_{i=1}^r \alpha_iG(z+h_i),$ where $G(z)=\exp(\log^{c+1}z),$ $z\in \mathbb{C}.$ This is well defined and analytic on $x>a,$ where $a=\max\{|h_i| : 1\leq i \leq r\}$ by extending its real-valued counterpart. For $N$ large enough so that the interval $[N,2N]$ lies on the semiplane where $F$ is defined, consider $\mathcal{C}$ to be the set of points of distance exactly $\frac{N}{2}$ from the real interval $[N,2N]$. Then, for $N \leq x \leq 2N$, by Cauchy's formula it is
\[
\frac{F^{(s)}(x)}{s!} = \int_{\mathcal{C}} \frac{F(\zeta)}{(\zeta-x)^{s+1}} \ d\zeta.
\]
We use the trivial bounds $|\zeta-x| \geq \frac{N}{2}$ for $\zeta \in \mathcal{C}$ and the fact that the length of $\partial C$ is no more than $(\pi+2)N$, whence
\[
\frac{F^{(s)}(x)}{s!} \leq 2\left(\frac{2}{N}\right)^s \max_{z \in \mathcal{C}} |F(z)|.
\]
Observe that 
\[
|F(z)| \leq C \exp(|\log^{c+1} (z+a)|),
\]
where $C:=\sum_{i=1}^r |\alpha_i|$, and $z \in \mathcal{C}$. Next, if $N$ is large enough in comparison to $a$ (which we can always ensure a priori), we can obtain exactly the same bound for the angle and logarithm as in \cite[Lemma~15]{bp}, guaranteeing that
\[
\exp(|\log^{c+1}(z+a)|) \leq G(3N)
\]
as long as $z \in \mathcal{C}$, which completes the proof.
\end{proof}

We need one more lemma to finally be able to prove Theorem~\ref{thm2aif98}. It follows using the same arguments as in Lemma~\ref{lem15aif98}. We need only to restrict the domain of $x$, but when $N$ is very large this has no impact in our further calculations.

\begin{lemma}\label{combinedineq}
Let $\alpha>0, 0 <\vartheta<1$ and put $\alpha=N^{-A}$. Let $a=\max\{ |h_i| : 1 \leq i \leq r\}$. Then for $N+a<x\leq 2N-a$ (with $N$ big enough), we have 
\begin{itemize}
    \item[(a)] Let $s \in \N$ such that
    \[
    s(1-\vartheta)+A \geq  \frac{s\log2}{\log N}+\frac{\log C}{\log N}+\frac{\log^{c+1} (3N)}{\log N},
    \]
    where $C = \sum_{i=1}^r |\alpha_i|$. Then,
    \begin{equation}\label{upperbound1}
    \left| \frac{\alpha F^{(s)}(x)}{s!}\right| \leq N^{-\vartheta s} \quad (N+a<x \leq 2N-a).
    \end{equation}
    \item[(b)] Let $s \in \N$ with $s \leq (c+1)\log^{c} N$ and let $d$ and $\tau$ be the constants from Proposition~\ref{lowerbound} which only depend on the $\alpha_j, h_j$ and $c$. If
    \[
    s(1-\vartheta)+A+\tau\leq \log^{\gamma-1} N+\frac{\log d}{\log N}.
    \]
    Then,
    \begin{equation}\label{lowerbound1}
    \left| \frac{\alpha F^{(s)}(x)}{s!}\right| \geq N^{-\vartheta s} \quad (N+a<x \leq 2N-a).
    \end{equation}
\end{itemize}
\end{lemma}
\begin{proof}
It will follow from our previous results in Proposition~\ref{lowerbound} and Lemma~\ref{upperbound}. More precisely: for part (a), note that, by Lemma~\ref{upperbound}, it is enough to show that
\[
\alpha \cdot C \cdot  G(3N) \cdot\left(\frac{2}{N}\right)^s \leq N^{-\vartheta s}.
\]
Taking logarithms and dividing by $\log N$ yields \eqref{upperbound1}.

For part (b), use Proposition~\ref{lowerbound} first, then notice that since $s+\tau \leq (c+1)\log^{c} N$, the function $x^{-(s+\tau)}G(x)$ is increasing (cf. the proof of \cite[Lemma~16]{bp}), whence it is enough to show that
\[
\frac{\alpha\cdot d \cdot G(N)}{N^{s+\tau}} \geq N^{-\vartheta s}.
\]
As in (a), taking logarithms and dividing by $\log N$ gives \eqref{lowerbound1}.
\end{proof}

We can now give the proof of Theorem~\ref{thm2aif98}.

\begin{proof}[Proof of Theorem~\ref{thm2aif98}]
It suffices to consider $\alpha>0$ by taking conjugates if necessary. First we take $0<\theta<1/3$ and pick $K<(c+1)(1-\theta)$, which can always be achieved. We will show that $\alpha F(x) \in \mathcal{F}(\theta, \frac{\theta}{2},\frac{3}{4},\delta)$, where
\[
\delta = \frac{1}{20}(1-\gamma)\cdot \left( \frac{1}{1-\theta}-\frac{1}{1-\frac{\theta}{2}} \right),
\]
and $\alpha \in [G(2N)^{-\gamma}, G(N)^K]$. We put $\Theta_0:=3/4,$ $\Theta :=\theta/2$, and $k:=\lfloor(9+C)\log^{c} N\rfloor+1$. Setting, as above, $\alpha=N^{-A}$, we can rewrite our choice for $\alpha$ as
\[
-K\log^{c} N \leq A \leq \gamma \frac{\log^{c+1} 2N}{\log N}.
\]

Consider $\vartheta=3/4$ and $s=k$, which ensures that the hypotheses for Lemma~\ref{combinedineq} (a) are satisfied. Then, it is 
\[
\left| \frac{\alpha F^{(k)}(x)}{k!}\right| \leq N^{-\frac{3}{4}k}
\]
for $N+a<x \leq 2N-a$, which is the first requirement that allows us to later apply Karacuba's Lemma~\ref{lem10aif98}. Next, we want to apply once again Lemma~\ref{combinedineq} in order to find the values of $s$ for which $N^{-\theta s} \leq \left| \frac{\alpha F^{(s)}(x)}{s!}\right| \leq N^{-\frac{1}{2}\theta s}$, for $N+a <x \leq 2N-a$. 

We will obtain the desired inequalities (thanks to Lemma~\ref{combinedineq} above) provided we have
\begin{equation}\label{inequalitySetS}
 \left(1-\frac{\theta}{2}-\frac{\log 2}{\log N}\right)^{-1}\left(\frac{\log^{c+1}(3N)+\log C}{\log N}-A\right)\leq s \leq \frac{1}{1-\theta} \left( \log^{c} N -A-\tau+\frac{\log d}{\log N}\right).
\end{equation}
We note that, with our choice of $\theta,$ the right inequality above for $s$ implies that $s \leq (c+1)\log^{c} N<k$ provided $N>d$. Thus, we may take the set $\mathcal{S}$ in Lemma~\ref{lem10aif98} to be the set of all $s$ such that \eqref{inequalitySetS} is satisfied. Given our bounds in $A$ afforded by our choice of $\alpha$, we see that if $N$ is large enough, then $|S| \geq \delta k$, so we can apply Lemma~\ref{lem10aif98} to conclude.
\end{proof}

\section{Proofs of main results}\label{sec: proofs of main}

This final section presents the proofs of our main results. Analogously to the previous sections, we let $G(x)=x^{\log^c x},$ $0<c<1/2.$

\subsection{Exponential sums} We can use Theorem~\ref{thm2aif98} to obtain Theorem~\ref{SPexpsum}.

\begin{proof}[Proof of Theorem~\ref{SPexpsum}]
Using the dyadic summation trick we can write: 
\[
\E_{n \in[N]} e \left( \sum_{j=1}^r \alpha_jG(n+h_j) \right) =\frac{1}{N}\sum_{k=0}^{\lfloor \log_2 N \rfloor +1}\sum_{\frac{N}{2^{k+1}}<n\leq \frac{N}{2^k}} e \left( \sum_{j=1}^r \alpha_j G(n+h_j) \right) .
\]
Next, Theorem~\ref{thm2aif98} allows us to estimate the inner dyadic sums (always for $N$ large enough, and taking into account that the intervals of the form $[K,2K]$ actually get shrunk into $[K+a,2K-a],$ which does not affect the average), yielding the upper bound
\[
\frac{1}{N}\sum_{k=0}^{\lfloor \log_2 N \rfloor +1}\frac{N}{2^{k+1}} e^{-\kappa\log^{1-2c} N} \leq e^{-\kappa\log^{1-2c} N} \sum_{k \geq 0} \frac{1}{2^{k+1}}=e^{-\kappa\log^{1-2c} N},
\]
which goes to $0$ as $N \to \infty$. This establishes the result.
\end{proof}

A useful consequence of Theorem~\ref{SPexpsum} is  the following result which follows the philosophy of \cite[Proposition~4.3]{Kouts21} and is the second ingredient for the proof of Theorem~\ref{T: equi}. 

\begin{proposition}\label{P: Weyl-type}
For every distinct $h_1,\ldots,h_r\in \mathbb{Z}$ and every $(\alpha_1,\ldots,\alpha_r)\in \mathbb{R}^r\setminus\mathbb{Z}^r,$ we have that
\begin{equation}\label{E: Weyl-type}
\lim_{N \to \infty}\E_{n \in [N]} e \left( \sum_{j=1}^r\alpha_j \left\lfloor G(n+h_j) \right\rfloor \right) = 0.
\end{equation}
\end{proposition}

\begin{proof} We present the $r=2$ case for convenience as it contains all the details for the general $r\in\mathbb{N}$ case. To do so we split the proof into three cases. 

\medskip

{\bf Case 1.} $\alpha_1,$ $\alpha_2\in \mathbb{R}\setminus\mathbb{Q}.$ 

\medskip

To show \eqref{E: Weyl-type} it suffices to show that $((\alpha_1\lfloor G(n+h_1)\rfloor, \alpha_2\lfloor G(n+h_2)\rfloor))_n$ is equidistributed in $\mathbb{T}^2$. To this end it suffices to show that the sequence 
\begin{equation}\label{quadrupleED}   
((\alpha_1 G(n+h_1), G(n+h_1), \alpha_2 G(n+h_2), G(n+h_2)))_n
\end{equation}
is equidistributed in $\mathbb{T}^4.$ Indeed, if we can show that the aforementioned quadruple is equidistributed in $\T^4$, we can then consider, given $(k_1,k_2) \in \Z^2\setminus \{(0,0)\}$ the function
\[
\psi(x,y,z,t):=\e(k_1(x-\alpha_1\{y\})+k_2(z-\alpha_2\{t\}))
\]
to obtain the equidistribution claim, using Weyl's criterion. Now, \eqref{quadrupleED} is equidistributed on $\T^4$ if and only if 
\begin{equation}\label{E: 1.1}
\left((a_1\alpha_1+a_2) G(n+h_1)+(a_3\alpha_2+a_4) G(n+h_2)\right)_n    
\end{equation}
is equidistributed in $\mathbb{T}$ for all $(a_1,a_2,a_3,a_4)\in \mathbb{Z}^4\setminus\{(0,0,0,0)\}.$ Since $\alpha_1,$ $\alpha_2\in \mathbb{R}\setminus\mathbb{Q},$ we have that the coefficients of the $G(n+h_i)$'s in \eqref{E: 1.1} are non-zero, so, the result follows by Theorem~\ref{SPexpsum}. 

\medskip

{\bf Case 2.} $\alpha_1,$ $\alpha_2\in \mathbb{Q}\setminus\mathbb{Z}.$

\medskip

If $\alpha_j=\frac{p_j}{q_j},$ we set $m=q_1q_2.$ Using some algebra, \eqref{E: Weyl-type} follows in this case if we show that \[\lim_{N\to \infty} \frac{1}{N}\sum_{n=1}^N e\left(\frac{1}{m}(k_1\lfloor G(n+h_1)\rfloor+k_2\lfloor G(n+h_2)\rfloor)\right)=0,\] for all $m\geq 2,$ $1\leq k_1, k_2\leq m-1.$

\medskip 

By Theorem~\ref{SPexpsum} we have that $((G(n+h_1)/m, G(n+h_2)/m))_n$ is equidistributed in $\mathbb{T}^2$ for all $m\geq 2.$ 
So, since $\lfloor x\rfloor=\lfloor\frac{x}{m}\rfloor m+j$ if $\frac{j}{m}\leq \left\{\frac{x}{m}\right\}\leq \frac{j+1}{m},$ $0\leq j\leq m-1,$ setting $E_{j_1,j_2}=[\frac{j+1}{m},\frac{j_1+1}{m})\times [\frac{j_2}{m},\frac{j_2+1}{m}),$ $0\leq j_1, j_2\leq m-1,$ we have:
	\begin{equation*}
	\begin{split}
	&\quad \lim_{N\to \infty} \frac{1}{N}\sum_{n=1}^N e\left(\frac{1}{m}(k_1\lfloor G(n+h_1)\rfloor+k_2\lfloor G(n+h_2)\rfloor)\right)
	\\&= \lim_{N\to \infty} \frac{1}{N}\sum_{n=1}^N\sum_{j_1,j_2=0}^{m-1} e\left(\frac{1}{m}\left(k_1\left(\left\lfloor \frac{G(n+h_1)}{m}\right\rfloor m+j_1\right)+k_2\left(\left\lfloor\frac{G(n+h_2)}{m}\right\rfloor m+j_2\right)\right)\right)
	\\
&  \quad\quad\quad\quad\quad\quad\quad\quad\quad\quad\quad\quad\quad\quad\quad\quad\quad\quad\quad\quad \times{\bf{1}}_{E_{j_1,j_2}}(\{G(n+h_1)/m\},\{G(n+h_2)/m\})
		\\&= \lim_{N\to \infty}\frac{1}{N}\sum_{n=1}^N\sum_{j_1,j_2=0}^{m-1} e\left(\frac{1}{m}(k_1 j_1+k_2 j_2)\right){\bf{1}}_{E_{j_1,j_2}}(\{G(n+h_1)/m\},\{G(n+h_2)/m\})
		\\&=\sum_{j_1,j_2=0}^{m-1} e\left(\frac{1}{m}(k_1 j_1+k_2 j_2)\right)\int_0^1\int_0^1 {\bf{1}}_{E_{j_1,j_2}}(x,y)\;dx\; dy
		\\&=\frac{1}{m^2}\sum_{j_1,j_2=0}^{m-1} e\left(\frac{1}{m}(k_1 j_1+k_2 j_2)\right)=0.
	\end{split}
    \end{equation*}

{\bf Case 3.} $\alpha_1\in \mathbb{R}\setminus\mathbb{Q}$ and $\alpha_2\in \mathbb{Q}\setminus\mathbb{Z}.$\footnote{ The case where $\alpha_2\in \mathbb{R}\setminus\mathbb{Q}$ and $\alpha_1\in \mathbb{Q}\setminus\mathbb{Z}$ is analogous.}

\medskip

Using some algebra, to show \eqref{E: Weyl-type} it suffices to show that, for all $m\geq 2,$ $1\leq k\leq m-1,$ 
\[\lim_{N\to\infty}\frac{1}{N}\sum_{n=1}^N e\left(\alpha_1\lfloor G(n+h_1)\rfloor+\frac{k}{m}\lfloor G(n+h_2)\rfloor\right)=0.\]
Following Cases 1 and 2, the previous relation holds if $((\alpha_1 G(n+h_1), G(n+h_1), G(n+h_2)/m))_n$ is equidistributed in $\T^3$ for all $m\geq 2.$ The latter follows by Theorem~\ref{SPexpsum}.
\end{proof}

\begin{proof}[Proof of Theorem~\ref{T: equi}] The equidistribution of the first sequence follows immediately by Theorem~\ref{SPexpsum} and of the second one by Proposition~\ref{P: Weyl-type}.
\end{proof}


\subsection{Characterization of Furstenberg systems and topological corollaries}\label{SEC: results1}
 We will use Theorem~\ref{SPexpsum} and Proposition~\ref{P: Weyl-type} to prove Theorem~\ref{mainthm}; a characterization of the Furstenberg systems of the sequences $a(n):=e(G(n))$ and $b(n):=e(\alpha[G(n)]),$ $\alpha \in \R \setminus \Z.$


\begin{proof}[Proof of Theorem~\ref{mainthm}]
We begin by proving the result for the sequence $a(n)$. By Corollary~\ref{cor: FsysBern}, it is enough to show that
\[
\lim_{N \to \infty} \E_{n \in [N]} e\left( \sum_{j=1}^r c_j G(n+h_i) \right) = 0
\]
for all distinct $h_1,\dots,h_j \in \Z$ and all $(c_1,\dots,c_r) \in \Z^r\setminus\{\vec{0}\}$. The equality above is now a consequence of the first part of Theorem~\ref{T: equi}, completing the proof.

Next, we move to the sequence $b(n)$. We distinguish two cases. First, if $\alpha \in \R \setminus \Q$, by Corollary~\ref{cor: FsysBern} it is enough to show that
\begin{equation}\label{eq: averages brackets}
\lim_{N \to \infty} \E_{n \in [N]} e\left( \sum_{j=1}^r \alpha c_j \left\lfloor G(n+h_j)\right\rfloor \right) = 0
\end{equation}
for all distinct $h_1,\dots,h_j \in \Z$ and all $(c_1,\dots,c_r) \in \Z^r\setminus\{\vec{0}\}$. Since $\alpha \in \R \setminus \Q$ this follows by the second part of Theorem~\ref{T: equi}. Next, if $\alpha \in \Q \setminus \Z$, put $\alpha=p/q$ in lowest terms. Then, we only really need to check the averages in \eqref{eq: averages brackets} above hold when $1 \leq |c_j| \leq q-1$ for all $j=1,\dots,r$, where, again, the result follows from Theorem~\ref{T: equi} effortlessly.
\end{proof}

As a consequence, we obtain Corollary~\ref{topcor intro}, using basic results on disjointness.

\begin{proof}[Proof of Corollary~\ref{topcor intro}]
Recall that $a(n)=e(G(n))$. Using its Furstenberg representation, we can write $a(n)=g(S^ny)$, where $S: (\S^1)^{\Z} \to (\S^1)^{\Z}$ is the shift map and $g(y)=y(0) \in C((\S^1)^{\Z})$ is the projection onto the $0$-th coordinate. Therefore, the limit in the left hand side of the first equation of \eqref{topcor} can be rewritten as
\[
\lim_{N \to \infty} \E_{n \in [N]} (f \otimes g)((T \times S)^n(x,y)),
\]
which is nothing but a joining of the the shift system $((\S^1)^{\Z}, \otimes_{n \in \Z} m_{\text{Leb}}, S)$ and the system $(X,T)$. It is classical that Bernoulli systems are disjoint from any system with $h(T)=0$, which means that the limit above is equal to
\[
\int_{(\S^1)^\Z} g \ d(\otimes_{n \in \Z} m_{\text{Leb}}) \int_Y f \ d\rho, 
\]
for some $T$-invariant measure $\rho$ on $(X,T)$. The first integral is $0$ as it is equal to the sum
\[
\lim_{N \to \infty} \E_{n \in [N]} a(n)=\lim_{N \to \infty} \E_{n \in [N]} e(G(n))=0,
\]
which completes the proof for the sequence $a(n)$. The proof for the sequence $b(n)$ is almost identical, so we skip the details.
\end{proof}

\subsection{A von Neumann type result}\label{SEC: results2}
In this last subsection we will give the proof of Theorem~\ref{T: vN}.

\begin{proof}[Proof of Theorem~\ref{T: vN}]
 Using Theorem~\ref{T:split}, we have $L^2(\mu)=L^2(\mu)_{\text{inv}}\oplus L^2(\mu)_{\text{erg}}.$ For $f\in L^2(\mu)_{\text{inv}}$ we have that $T_1^{\lfloor G(n+h_1)\rfloor}\cdots T_r^{\lfloor G(n+h_r)\rfloor}f=f,$ so, it suffices to show that for $f\in L^2(\mu)_{\text{erg}}$ we have 
\[\lim_{N\to\infty}\norm{\frac{1}{N}\sum_{n=1}^N T_1^{\lfloor G(n+h_1)\rfloor}\cdots T_r^{\lfloor G(n+h_r)\rfloor}f}_2=0.\] This follows from Theorem~\ref{T:BH} and Proposition~\ref{P: Weyl-type}. Indeed,
\begin{equation*}
\begin{split}
	&\quad
\norm{\frac{1}{N}\sum_{n=1}^N \ T_1^{\lfloor G(n+h_1)\rfloor}\cdots T_r^{\lfloor G(n+h_r)\rfloor}f}^2_2 
\\&= \frac{1}{N^2}\sum_{n,m=1}^N \left\langle \left(\prod_{j=1}^r T_j^{\lfloor G(n+h_j)\rfloor-\lfloor G(m+h_j)\rfloor}\right)f,f\right\rangle   \\&= \frac{1}{N^2}\sum_{n,m=1}^N \int_{\mathbb{T}^r} e\left(\sum_{j=1}^r \alpha_j(\lfloor G(n+h_j)\rfloor-\lfloor G(m+h_j)\rfloor)\right)\;d\nu_f(\alpha_1,\ldots,\alpha_r)
\\&= \int_{\mathbb{T}^r} \left| \frac{1}{N}\sum_{n=1}^N e\left(\sum_{j=1}^r \alpha_j\lfloor G(n+h_j)\rfloor\right)\right|^2 \;d\nu_f(\alpha_1,\ldots,\alpha_r)
\end{split}
\end{equation*}
which goes to $0$ as $N\to \infty$ since $f\in L^2(\mu)_{\text{erg}},$ whence, by Corollary~\ref{spectralET} we have $\nu_f(\{(0,\ldots,0)\})=0.$ 
\end{proof}

\bibliography{library}
\bibliographystyle{plain}

\end{document}